\documentclass{amsart}
\usepackage{amssymb}
\usepackage{amsmath}
\usepackage{color}
\usepackage{graphicx}
\usepackage{dsfont}

\newtheorem{theo}{\indent Theorem}[section]
\newtheorem{prop}[theo]{\indent Proposition}

\newtheorem{ex}[theo]{\indent Example}
\newtheorem{ass}[theo]{\indent Assumption}

\def\D{{\mathbb D}}
\def\bM{{\mathbf{M}}}

\def \E{I\!\!E}
\def \P{I\!\!P}

\newcommand{\R}{\mathbb {R}}

\date{September 13, 2019}
\author{A. Galves \and E. L\"ocherbach  \and C. Pouzat \and E. Presutti}

\begin{document}
\title{A system of interacting neurons with short term synaptic facilitation.}

\address{A. Galves: Universidade de S\~ao Paulo.}

\email{galves@usp.br}

\address{E. L\"ocherbach: Universit\'e de Paris 1 Panth\'eon Sorbonne.}

\email{eva.locherbach@univ-paris1.fr}

\address{C. Pouzat: Universit\'e Paris Descartes.}

\email{christophe.pouzat@parisdescartes.fr}

\address{E. Presutti: Gran Sasso Science Institute.}

\email{errico.presutti@gmail.com}


\maketitle

\begin{center}
{\bf Abstract}
\end{center}
In this paper we present a simple microscopic stochastic model
describing short term plasticity within a large homogeneous network of
interacting neurons. Each neuron is represented by its membrane
potential and by the residual calcium concentration within the cell at
a given time. Neurons spike at a rate depending on their membrane
potential. When spiking, the residual calcium concentration of the
spiking neuron increases by one unit. Moreover, an additional amount
of potential is given to all other neurons in the system. This amount
depends linearly on the current residual calcium concentration within
the cell of the spiking neuron. In between successive spikes, the
potentials and the residual calcium concentrations of each neuron
decrease at a constant rate.

We show that in this framework, short time memory can be described as
the tendency of the system to keep track of an initial stimulus by
staying within a certain region of the space of configurations during
a short but macroscopic amount of time before finally being kicked out
of this region and relaxing to equilibrium. The main technical tool is
a rigorous justification of the passage to a large population limit
system and a thorough study of the limit equation.

\maketitle
\section{Introduction}
In this paper we present a simple microscopic stochastic model describing short term plasticity within a large network of interacting neurons. In this framework it is possible to describe short time memory of the system in a precise mathematical way. Namely, short time memory can be seen as the tendency of the system to keep track of an initial stimulus by staying within a certain region of the space of configurations during a short but macroscopic amount of time before finally being kicked out of this region and relaxing to equilibrium.

In our model, the successive times at which the neuron emits an action potential are described by a point process. The stochastic spiking intensity of the neuron, i.e., the infinitesimal probability of emitting an action potential during the next time unit, conditionally on the past, depends on the past history of the neuron and it is affected by the activity of  other neurons in the network, either in an excitatory or an inhibitory way. 

Short term synaptic plasticity (STP) refers to a change in the synaptic efficacy on timescales which are of the order of milliseconds, that is, comparable to the timescale of the spiking activity of the network. We express this through the fact that the stochastic spiking intensity also depends on the synaptic efficacy of the neuron at that time. This synaptic efficacy changes over time as a function of the residual calcium concentration within the cell. In our model the residual calcium concentration increases by one unit any time the neuron spikes and decreases at a constant rate in between successive spikes. 

Since at least the last two decades, many papers have been devoted to STP. Probably starting with Markram and Tsodyks (1996) \cite{markram96} and Tsodyks et al. (1998)\cite{tm2},  a lot of these papers propose relatively simple phenomenological models and study, mostly numerically, their properties. Kistler and van Hemmen (1999) \cite{van}  consider a deterministic model which is an adaptation of the model of Tsodyks and Markram (1997) \cite{tm}  to the spike response model. They work within a homogenous strongly connected network and study the impact of STP on the stability of limit cycles.  Our model, though stochastic, is close to this. Several recent papers are devoted to the study of the effect of STP on working memory, see Barak and Tsodyks (2007) \cite{BT}, Mongillo et al. (2008) \cite{Mongillo} and the recent article by Seeholzer et al. (2018) \cite{see}. Finally, for a recent survey on STP, we refer the interested reader to the Scholarpedia article \cite{scholar}, and for a rather complete survey on the biological aspects, to Zucker and Regehr (2002) \cite{zucker}. 

Our model can be seen as a huge system of interacting pairs of coupled Hawkes processes. Hawkes processes provide good models for systems of spiking neurons by the structure of their intensity processes and have been widely studied, see for instance Chevallier et al.\ (2015) \cite{ccdr}, Chornoboy et al.\  (1988) \cite{chorno},  Hansen et al.\ (2015)  \cite{hrbr}, Reynaud-Bouret et al. (2014) \cite{pat} and Ditlevsen and L\"ocherbach (2017) \cite{SusanneEva}. 

In our model we make the following basic mathematical assumptions. First of all, we work within a mean-field system in which each neuron interacts with all other neurons in a homogenous way. Second, we only consider excitatory synapses. Finally, our model of STP only describes facilitation, not depression. In this framework we study  in a rigorous way the intermediate time behavior of the process. This is the content of our main result,  Theorem \ref{theo:4}. The main step of the proof of this theorem is a rigorous justification of the passage to a large population limit model. 

Describing short term memory as the tendency of the system to stay within a certain region of the state space representing some initial stimulus is not a new idea, neither is the idea that a proof of this behavior should be done through the analysis of the limit system. Similar ideas appear already in Barak and Tsodyks (2007) \cite{BT}, Mongillo et al. (2008) \cite{Mongillo} and Seeholzer et al. (2018) \cite{see}, see also the recent paper by Schmutz et al. (2018) \cite{Schmutz}.  Nevertheless to the best of our knowledge our paper is the first in which these results are rigorously mathematically proved. 

{\bf Organisation of the paper.}
This paper is organised as follows. In Section 2, we introduce our model and state the main results of the paper. This section is complimented by a simulation study. The proofs are given in Sections 3--7.

\section{Overview of the paper.} 
\subsection{Notation}
The following notation is used throughout the paper. 
\begin{itemize}
\item If $ M $ is a counting measure on $\R_+ \times \R_+ , $ we shall use the following notation for integration over semi-closed boxes. $ \int_0^t \int_0^\infty f(s, z ) M(ds, dz ) = \int_{\R_+} \int_{\R_+ } 1_{ ] 0, t ] }  (s) f(s, z ) M( ds, dz) ,$ for any positive measurable function $f.$  
\item If $ g $ is a probability measure on $ (\R_+ , {\mathcal B} ( \R_+) )   $ and $ \int_{\R_+} |f| d g < \infty,  $  then we write $ \int_0^\infty f(r) g (dr) $ for the integral $ \int_{\R_+} f d g .$ 
\item For any bounded function $ f : \R_+ \to \R  $ we write $\| f\|_\infty = \sup_{ x \in \R_+}  |f(x) |.$
\end{itemize}

\subsection{Description of the model}
We consider a system of $N$ interacting neurons with membrane potentials $ (U_t^N (1), \ldots, U^N_t (N) ) $ together with the  corresponding residual calcium concentration within each neuron  $ (R_t^N (1), \ldots , R^N_t(N) ). $ 
Each neuron $i,$ independently of the others, spikes at rate $ \varphi (U^N_{t-} ( i ) ).$ When spiking, it gives an additional amount $ \alpha R_{t-}^N (i )  / N $ of potential to all neurons. ($R$ is dimensionless  and $ \alpha $ is a potential). $\alpha > 0$ is a measure of the interaction strength, and the interaction is modulated by the current value of the calcium concentration of the spiking neuron. At the same time, the residual calcium concentration of the spiking neuron is increased by $ 1.$ This models the short term plasticity. In between successive spikes, the potential of each neuron decreases at rate $ \beta > 0 ,$ and their residual calcium concentrations decrease at constant rate  $\lambda > 0 .$ $\beta $ and $\lambda$ are homogeneous to the inverse of a time. 

To define the process, consider a family of i.i.d.\ Poisson random measures
$(\bM^i(ds,  dz))_{i\geq 1 }$ on $\R_+  \times  \R_+ $ having intensity measure $ds dz$ each. Here, $ds$ is homogeneous to a time and $dz $ to the inverse of a time. Finally we consider an i.i.d.\ family $(U^{N}_0 (i ) , R^{N }_0 (i )  )_{i=1,\dots,N}$ of $\R^2 _+ $-valued random variables, 
independent of
the Poisson measures and distributed according to some probability measure $ \eta_0 ( du, dr) $ on $\R_+^2.$ 
Then the system of interaction neurons is represented by the Markov process
$(U^N_t, R^N_t)  = (U^{N}_t (1) , \ldots , U^{N}_t (N) , R^{N }_t (1) , \ldots , R^{N}_t (N)   )$
taking values in $\R_+^{2N} $ and solving, for $i=1,\dots,N$, for $t\geq 0$,
\begin{eqnarray}\label{eq:dyn}
U^{N}_t (i )  &= & U^{N}_0 (i)  - \beta \int_0^t  U^{N}_s (i)    ds + \frac{\alpha}{N}\sum_{ j =1 }^N \int_0^t  \int_0^\infty  R^{N}_{s-} (j)  1_{ \{ z \le  \varphi ( U^{N}_{s-} (j) ) \}} \bM^j (ds,  dz) ,   \\
R^{N }_t (i)  &=& R^{N }_0 (i )  - \lambda \int_0^t  R^{N}_s (i )    ds  +   \int_0^t \int_0^\infty 
  1_{ \{ z \le  \varphi  ( U^{N}_{s-} (i ) ) \}} \bM^i (ds, dz) . 
\nonumber
\end{eqnarray}   
The coefficients of this system are the positive constants $ \alpha , \beta , \lambda > 0 $ together with the spiking rate function $ \varphi . $  The generator of the process $(U_t^N, R_t^N)$   
is given for any smooth test function $f : \R_+^{ 2N} \to \R $ by 
$$
A^N f( u, r ) = - \sum_{ i=1}^N \left( \beta u^i \partial_{u^i } f (u, r) + \lambda r^i \partial_{r^i} f( u, r) \right) + 
 \sum_{i=1}^N \varphi ( u^i ) \left[ f ( \pi^i ( u, r) ) - f (u, r ) \right] ,
$$
where
$$ \pi^i ( u, r)  (j) := ( u^j + \alpha r^i / N ,r^j) \mbox{ if } j \neq i , \pi^i  ( u, r ) (i) := (u^i + \alpha r^i /N , r^i + 1 ) .$$
Notice that \eqref{eq:dyn} is close to the system studied in \cite{van}, when taking the Heaviside function as a spiking rate (that is, spiking does only occur when reaching a fixed deterministic threshold, but when hitting this threshold, it occurs with certainty, that is, at rate $ = \infty $). 
In the present paper, following the classical reference Brillinger and Segundo (1979) \cite{Brillinger}, we will however suppose most of the time that 
\begin{ass}\label{ass:1}
$\varphi : \R_+ \to \R_+ $ is bounded and Lipschitz continuous with Lipschitz constant $L_\varphi .$ Moreover we have $ \varphi ( 0) = 0,$ $\varphi ( x) > 0 $ for all $ x > 0.$ 
\end{ass}
Under minimal regularity assumptions on the spiking rate, if we work at a fixed system size $N,$ this process will die out in the long run as shows the following
\begin{theo}\label{prop:dieout}
Grant Assumption \ref{ass:1}. If $ \varphi $ is differentiable in $ 0, $ then the system stops spiking almost surely. As a consequence, the unique invariant measure of the process $(U_t^N, R_t^N ) $ is given by  the Dirac measure $\delta_{(\bf 0 , \bf 0)}, $
where $\bf 0 \in \R^N $ denotes the all-zero vector in $\R^N .$ 
\end{theo}
This situation might however change if we consider large-population limits of the system. 
\subsection{Large population limits}
In Section \ref{sec:convergence} we show that 
the solution $(U^N_t, R^N_t)_{t\geq 0}$ behaves, for $N$ large, as $N$ independent copies
of the solution $(U_t, R_t)_{t\geq 0}$ of the following {\it nonlinear}, in the sense of McKean, SDE 
\begin{eqnarray}\label{eq:dynlimit}
U_t &=& U_0 - \beta \int_0^t U_s ds + \alpha \int_0^t   \E[\varphi (U_s) R_s ]ds, \\
R_t &=& R_0 - \lambda \int_0^t R_s ds +   \int_0^t \int_0^\infty 
  1_{ \{ z \le  \varphi  ( U_{s-}) \}} \bM (ds, dz) \nonumber .
\end{eqnarray} 
In the above formula, $(U_0, R_0)$ is an $\eta_0$-distributed random variable, independent of a Poisson measure $\bM (ds,dz)$ on $\R_+ \times \R_+ $ having intensity measure $ ds dz$. Concerning the law $ \eta_0$ of the initial condition, in the sequel we impose 
\begin{ass}\label{ass:2}
$ \eta_0 (du, dr ) = \delta_{u_0 } (du ) g_0 ( dr ) , $ for some fixed $ u_0 > 0.$ Here, $ g_0 $ is a probability measure on $\R_+ $ such that  $ \int_0^\infty r^2 g_0 (dr) < \infty .$   
\end{ass}
In what follows, under Assumption \ref{ass:2}, we shall write $ r_0 := \int_0^\infty r g_0 (dr) . $ 
Under this condition, our Theorems \ref{theo:coupling} and \ref{theo:couplingbis} in Section 5 below provide an explicit coupling showing that the finite  system is  close to the limit system.

\subsection{Modeling short term memory.}
For smooth spiking rate functions $ \varphi , $ by Proposition \ref{prop:dieout},  the finite size system has only one invariant state corresponding to extinction of the system. In the large-population limit however, this situation changes for suitable choices of the form of the spiking rate function. More precisely, we suppose that
\begin{ass}\label{ass:psi}
$\varphi $ is non-decreasing and differentiable with $  \varphi ' (0) > 0  . $
Moreover, there exists a constant $D > 0 $ such that the equation  $ D \varphi^2  ( x) = x $ possesses exactly three solutions $ x_0 = 0 < x_1 <x_2 $ in $[ 0, \infty [  .$ ($D$ is homogeneous to a potential times a time squared.) 
\end{ass}

Let us write $ u^{max} = x_2 ,  r^{max} = \frac{1}{\lambda} \varphi ( u^{max} ) .$ We show in Proposition \ref{lem:attr} below that $ (u^{max}, r^{max} ) $ is an attracting equilibrium of the limit system \eqref{eq:dynlimit}, for suitable choices of $ \alpha, \beta $ and $\lambda.$

Suppose now we observe a huge system of interacting neurons which is undergoing synaptic plasticity modulated by the residual calcium concentrations within each neuron. Hence, $N$ is big without being infinite. 
We expose the system to some initial stimulus pushing it into the vicinity of the attracting non-trivial equilibrium point $ (u^{max} , r^{max} )$ of the limit system. At time $0, $ this stimulus is switched off, and we start observing the system, evolving according to \eqref{eq:dyn}. Since this point is attracting and $N$ large, the system is attracted to a small neighbourhood of $ (u^{max} , r^{max} ) $ and stays in this neighbourhood for a while.  We interpret this transient behavior as an expression of {\it short term memory}. Of course, in the long run, the system will finally get kicked out of this neighbourhood and start rapidly decaying towards the 
all-zero state. These ideas are formalised in the following theorem.

\begin{theo}\label{theo:4}
Grant Assumptions \ref{ass:1} and \ref{ass:2} and suppose moreover that $ \int_0^\infty e^{ u r} g_0 ( dr ) < \infty $ for some $ u > 0 .$ Fix some $T> 1.$ 

1. There exist positive constants $C_T $ and $c$ where $ C_T$ depends only on the parameters of the model and on $T$ and $c$  only on the parameters of the model with the following properties. For any $\varepsilon > 0 ,$ there exists $N_0 $ such that for all $ N \geq N_0, $ for every $1 \le i \le N ,$ 
 \begin{equation}
 \P \left( \sup_{s \le T} \left( | U_s^{N } (i ) - U_s |  + | \frac1N \sum_{j=1}^N R_s^{N } (j )  - \E R_s | \right)\geq  N^{-1/5} \varepsilon  \right) \le \\
C_T e^{ - c  \varepsilon^2 \sqrt{N} } .
\end{equation} 
2. Grant moreover Assumption \ref{ass:psi} and suppose that 
\begin{equation}\label{eq:constraint}
 \alpha  \geq D \beta \lambda . 
\end{equation} 
Then the equation $ x = \frac{\alpha}{\beta \lambda } \varphi^2 (x) $ possesses three solutions $ x_0 = 0 < x_1 < x_2 .$ We put $
u^{max} = x_2 ,  r^{max} = \frac{1}{\lambda} \varphi ( u^{max} ) .$ The points $(0,0)$ and $ (u^{max}, r^{max} ) $ are locally attracting equilibria of  the dynamical system 
\begin{equation}\label{eq:dynlimitbisbis}
 u_t =u_0   - \beta \int_0^t u_s ds +  \alpha \int_0^t \varphi ( u_s) r_s ds, \; r_t =r_0 - \lambda \int_0^t r_s ds + \int_0^t \varphi (u_s) ds  .
\end{equation}
3. Suppose $ (u_0, r_0 ) $ belongs to the domain of attraction of $ (u^{max}, r^{max} ) $ and 
$ U^N_0 (i ) = u_0, $ $ R^{N }_0 (i )  = r_0 $ for all $ i,N .$ Finally let, for $ \varepsilon > 0 ,$ 
$$ t_1 = t_1 ( \varepsilon ) = \inf \{ t : | u_t - u^{max}|+ | r_t -  r^{max}  | < \varepsilon  \}  .$$ 
Then for all $ N \geq N_0, $  for all $ 1 \le i \le  N,  $ 
$$ \P (|  U^N_{t_1} (i ) - u^{max }| \geq  2  \varepsilon    \mbox{  or  }  | \frac1N \sum_{j=1}^N R_{t_1}^{N } (j)  - r^{max} | \geq 2  \varepsilon ) \le C_{t_1} e^{ -c \varepsilon^2   \sqrt{ N}}.$$ 
\end{theo} 

\subsection{An example with a simulation study} 
We consider spiking rate functions of sigmoid type which are defined in terms of a parameter  $a  > 1  $ satisfying
$ 4a  < 1 + e^{ a}   $ 
by
$$ \varphi ( x) = \frac{4a}{ 1 + e^{-(x -a)}} - \frac{4a}{ 1 + e^{ a}} , x \geq 0 .$$
The point $ a$ is the inflexion point of $ \varphi,  $ and it is easy to see that there exist   $ x_1 \in ] 0, a [  $ and  $ x_2  \in ] a, \infty [ $ with $ \varphi^2  ( x_i ) = x_i , i =1, 2 .$ Thus, Assumption \ref{ass:psi} is satisfied with $D=1.$ 
\begin{figure} 
\begin{center}
   \includegraphics[width=15cm]{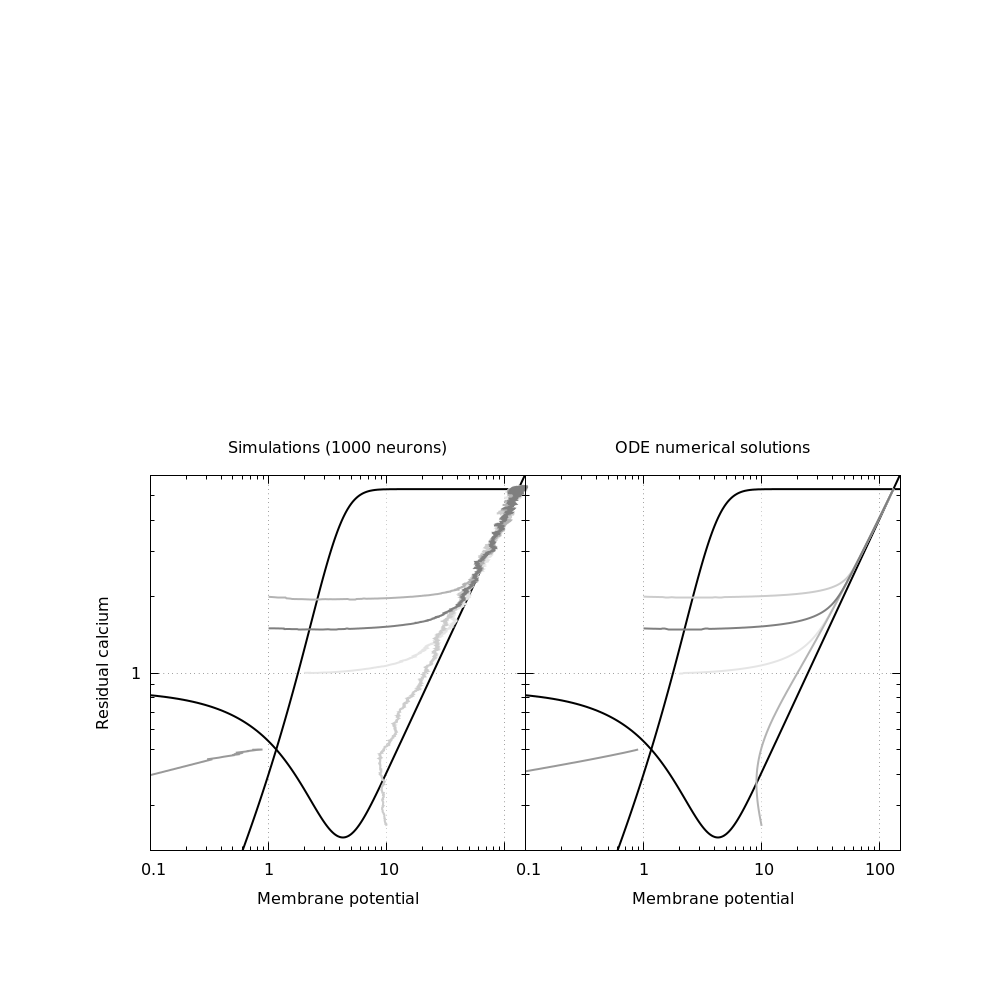} 
\end{center}
\caption{Phase plots on a log-log scale. Left, 5 trajectories (gray lines) of the mean residual calcium versus the mean membrane potential obtained by simulating a network of 1000 neurons from 5 different initial states (see the main text for simulation details). Right, 5 trajectories solutions of the limit ODE system \eqref{eq:dynlimitbisbis} with the same parameters and initial states (gray lines). On both plots the black curves show the null-cline of the mean membrane potential (V shaped) and of the mean residual calcium (inverted L shape).}
\label{fig1}
\end{figure}

Figure \ref{fig1} illustrates 5 trajectories of the mean residual calcium versus the mean membrane potential obtained by simulating a network of 1000 neurons from 5 different initial states on the left side. The null-clines corresponding to a null membrane potential derivative (V shaped) and a null residual calcium derivative (inverted L shaped) are also drawn. On the right side the numerical solution of the corresponding ODE system \eqref{eq:dynlimitbisbis} with the same initial conditions as the right side are shown.
A custom developed C code implementing Ogata's thinning method (see \cite{ogata:81}) was used for the simulations with the Xoroshiro128+ pseudo-random number generator of Blackman and Vigna, see \cite{blackman.vigna:18}. The ODE system \eqref{eq:dynlimitbisbis} was numerically solved using the ode program of the open source GNU plotutils package (https://www.gnu.org/software/plotutils/). The default method -- Runge-Kutta-Fehlberg with adaptive time steps -- was used. The network parameters were: $\alpha=107.78$; $\beta=50$; $\lambda=2.16$ (rounded to the second decimal); $a=3$. The parameters $ \alpha , \beta , \lambda $ are chosen such that $ \alpha = D \beta \lambda .$ The specific choice of $ \alpha , \beta , \lambda $ was arbitrary and guided by aesthetic reasons. For the network simulations the initial states were obtained by drawing the membrane potential and the residual calcium of each neuron from a uniform distribution centred on a user set mean value with a range set to 10\% of the mean. The (membrane potential, residual calcium) pairs were: (2,1); (1,2); (10,0.25); (0.75,0.5); (1,1.5). The initial values of the ODE numerical solution were set to these mean values. The phase plots shown on Figure \ref{fig1} use a log-log scale. The trajectory starting from (0.75,0.5) moves towards the origin: the network activity dies quickly in that case. All the other trajectories reach quickly (in less than 5 time units) the fixed point corresponding to the upper-right intersection of the 2 null-clines.

\begin{figure} 
\begin{center}
   \includegraphics[width=16cm]{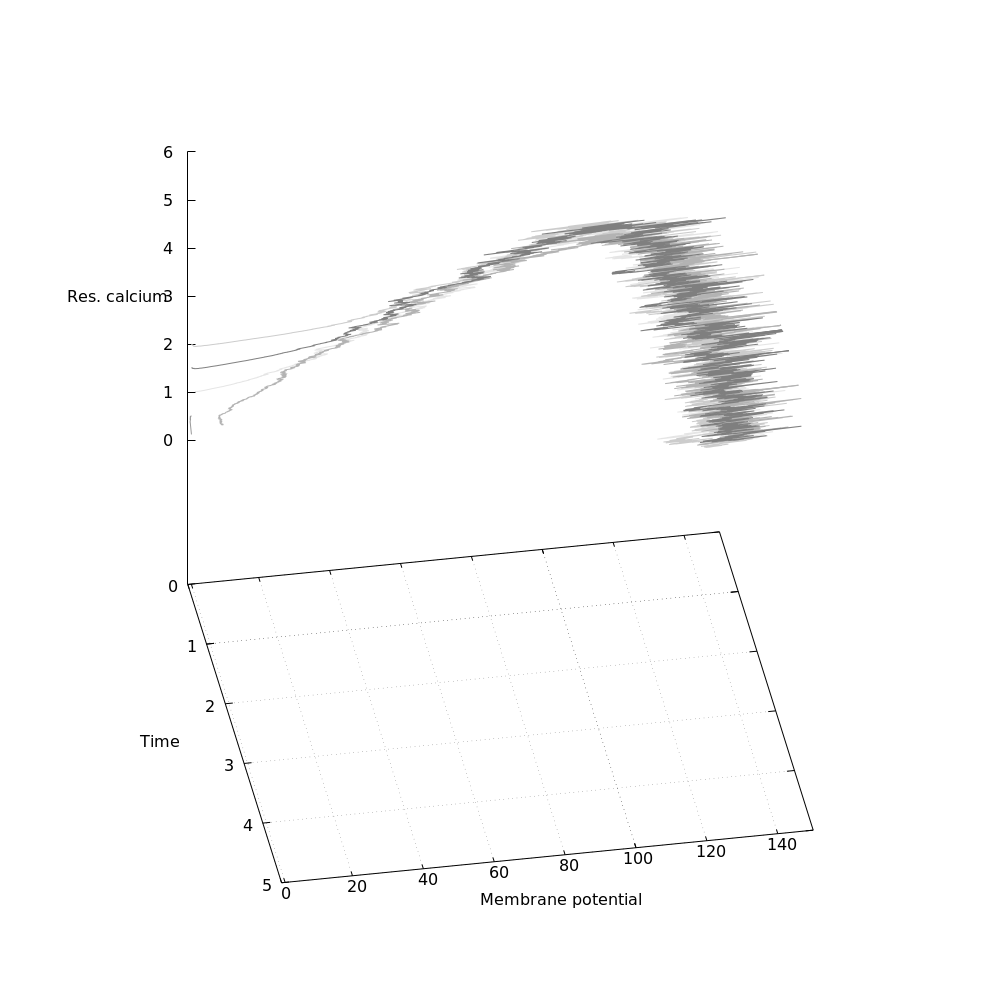} 
\end{center}
\caption{3D plot of 5 trajectories of the mean residual calcium and  the mean membrane potential obtained by simulating a network of 1000 neurons from 5 different initial states.}
\label{fig2}
\end{figure}
Figure \ref{fig2} shows the simulated trajectories of same network in 3D using linear scales. All the codes and instructions required to reproduce these simulations and figures can be found at the following address: {\tt https://plmlab.math.cnrs.fr/xtof/interacting\_neurons\_with\_stp}.

\subsection{Constants} In the whole paper, $C$ stands for a (large) finite constant
and $c$ stands for a (small) positive constant. Their values may change from line to line.
They are 
allowed to depend only on $\varphi , \alpha , \beta , \lambda $ and $\eta_0$, the law of the initial condition. Any other dependence will be 
indicated in subscript. For example, $C_T$ is a finite constant depending only on
$\varphi , \alpha , \beta , \lambda, \eta_0 $ and $T.$ The letter  $ K$ will be reserved to denote a bound on $ \| \varphi \|_\infty  .$

\section{Well-posedness of the particle system and proof of Theorem \ref{prop:dieout}} 
\begin{prop}
Under Assumption \ref{ass:1} a path-wise unique strong Markov process $ (U^{N}_t , R^{N }_t )$ exists which is solution of \eqref{eq:dyn} for all $ t \geq 0.$ 
\end{prop}

\begin{proof}
This follows from Theorem 9.1 in Chapter IV of Ikeda and Watanabe (1989) \cite{IW}.
\end{proof}

We now proceed with the 

\begin{proof}[Proof of Theorem \ref{prop:dieout}]
The proof works along the lines of the proof of Theorem 2.3 of Duarte and Ost (2016) \cite{do}. 

Firstly, it is immediate to show that almost surely the process comes back to a compact set infinitely often. This follows from the boundedness of $\varphi $ and the back driving force induced by the two drift terms $ - \beta U_t^{N } (i)  dt $ and $ - \lambda R^{N }_t (i )  dt .$ 

Let then $T_1 $ be the first jump time of the system. Then, as in Proposition 3.1 of \cite{do}, for any $ c > 0, $ 
$$ \inf_{ u_0 \in [0, c ]^N } \P ( T_1 = \infty | U^N_0 = u_0 ) \geq e^{- \frac{N}{\beta} \int_0^c \frac{\varphi (x) }{x} dx } > 0.$$
Since the process $U_t^N $ almost surely comes back to a suitable compact set $ [0, c]^N , $ the assertion then follows by a Borel-Cantelli argument. The details are in \cite{do}.  
\end{proof}

\section{Well-posedness  of the limit equation}
This section is devoted to the study of existence and uniqueness of the limit equation \eqref{eq:dynlimit}. A proof of the well-posedness of the limit system \eqref{eq:dynlimit}
\begin{eqnarray*}
U_t &=& U_0 - \beta \int_0^t U_s ds + \alpha \int_0^t   \E[\varphi (U_s) R_s ]ds, \\
R_t &=& R_0 - \lambda \int_0^t R_s ds +   \int_0^t \int_0^\infty 
  1_{ \{ z \le  \varphi  ( U_{s-}) \}} \bM (ds, dz) \nonumber ,
\end{eqnarray*}  
is not immediate due to the presence of the product term $  \E[\varphi (U_s) R_s ]ds $ in the first line of the above system leading to non-Lipschitz terms. Our Assumption \ref{ass:2} has been introduced to cope with this problem. Indeed, observe that under Assumption \ref{ass:2}, the limit process $U_t$  is a deterministic process, and we shall write
$ U_t = u_t $ to highlight this fact. Notice that in this case, the spike counting process of a typical neuron in the limit population
$$ t \to \int_0^t  \int_0^\infty 1_{\{ z \le \varphi (u_s  ) \}} \bM ( dz, ds ) $$
is an inhomogeneous Poisson process of rate $ \varphi ( u_t ) $ at time $t.$ 

\begin{prop}\label{prop:uniquelimit}
Grant Assumption \ref{ass:2} and suppose only that $\varphi $ is Lipschitz, satisfying $ \varphi( x) \le C \sqrt{x} $ for all $ x \geq x_0,$ for some  fixed  $ x_0 > 0 .$ Then a path-wise unique process $ (u_t , R_t )$ exists which is solution of
\eqref{eq:dynlimit}.
\end{prop}

\begin{proof}
Writing $ r_t = \E R_t , $ the system $(u_t, r_t)$ solves, since $ U_0 = u_0 $ is deterministic,
$$ d u_t = - \beta u_t dt + \alpha \varphi (u_t) r_t dt , d r_t = - \lambda r_t + \varphi (u_t) dt. $$ 
We first show that any solution is non-exploding. For that sake, suppose w.l.o.g. that $u_0 > x_0 .$ 
Since $ \varphi (x) \le C \sqrt{x} $ for all $ x \geq x_0, $ we use the change of variables $ g_t := \sqrt{ u_t} $ for all $ t \le \tau := \inf \{ s : u_s \le x_0 \} .$ Then, for all $ t \le \tau, $  
$$ d g_t \le - \frac{\beta}{2} g_t dt + C \frac{\alpha}{2} r_t dt , d r_t \le  - \lambda r_t + C g_t dt .$$ 
The rhs of the above inequality defines a linear equation which can be solved explicitly implying the existence of a non-exploding solution. 

To prove uniqueness of the solution, suppose that $ (u_t', R_t' ) $ is another solution, starting from the same initial conditions. Due to the first part of the proof, we know that for any 
$ T> 0, $ the functions $ r_t = \E R_t , $ $r_t'  = \E R_t' , $ $ \varphi (u_t ) $ and $ \varphi (u_t' ) $ are bounded on $ [0, T ] ,$ say by a constant $K_T . $ Then, by the Lipschitz continuity of $ \varphi $ with Lipschitz constant $L_\varphi, $   
$$| u_t - u'_t | \le \beta \int_0^t | u_s - u'_s | ds + \alpha L_\varphi  K_T \int_0^t | u_s - u'_s |ds  + \alpha K_T  \int_0^t \E  | R_s- R'_s| ds  $$
and 
$$ \E | R_t - R'_t| \le \lambda \int_0^t \E | R_s - R'_s| ds + L_\varphi \int_0^t  | u_s - u'_s | ds ,$$
implying that
$$  \E | R_t - R'_t| +| u_t - u'_t | \le C_T \int_0^t \left[  \E | R_s - R'_s| +| u_s - u'_s | \right] ds ,$$
and thus $ R_t = R'_t $ almost surely and $u_t = u'_t $ for all $t \le T,$ whence the uniqueness of the solution. 
\end{proof}

\section{Convergence of the particle system to the limit equation}\label{sec:convergence}
We now show that under Assumptions \ref{ass:1} and \ref{ass:2} the finite system \eqref{eq:dyn} converges to the limit equation \eqref{eq:dynlimit} in a certain sense. Recall that 
we suppose that $ (U^N_0 (i ) , R^N_0 (i )  )_{1 \le i \le N}  $ are i.i.d.\ $ \eta_0-$distributed random variables.

\subsection{A priori bounds}
In the sequel we shall use  a priori upper bounds on the processes of residual calcium concentrations. Recall that in this part of the paper we work under the assumption that $ \varphi $ is bounded and that $K=  \| \varphi \|_\infty  .$ 
We introduce
$$J_t (i )  :=\int_0^t \int_0^\infty 1_{ \{ z \le K \}} \bM^i (ds, dz)   , 1 \le i \le N.$$ 
Then $J_t (i)  , i \geq 1, $ are i.i.d. standard Poisson processes of rate $K$ each, and each $R_t^{N } (i )  $  
is stochastically dominated by 
\begin{equation}\label{eq:barrt}
R_t^N (i ) \le \check R^N_t (i)  :=   e^{- \lambda t} R_0^{N} (i )  + e^{- \lambda t} \int_0^t e^{\lambda s} d J_s (i )  .
\end{equation}
For each $ 1 \le i \le N, $ the process $ \check R^N (i )  $ is a Markov process with generator 
$$ \check A g ( r) = - \lambda r g' ( r) + K   ( g(r+1)- g(r) ) ,$$
for sufficiently smooth test functions $g.$ Moreover, the processes $ \check R^N (i ), 1 \le i \le N, $ are i.i.d. Taking a Lyapunov function $ V(x) = |x|, $ $ V(x) = x^2 $ or $ V(x) = e^{ux}, $ respectively, it is easy to see that for any
such choice of $V$ there exist suitable positive constants $c, d , r_0 ,$ such that  
$$  \check A V (r)  \le  c 1_{\{ r \le  r_0\}  }  - d  V(r)  .$$
A standard Lyapunov argument then implies that for all $ t \geq 0, $ 
$$ \E_r V ( \check R^N_t (i ) ) \le e^{ - d t } r + \frac{c}{d} ,$$ and hence
\begin{equation}\label{eq:tobeusedlater}
\sup_{ t \geq  0 } \E V(R^N_t (i) ) \le \sup_{ t \geq  0 } \E V(\check R^N_t (i) ) < \infty, 
\end{equation}
provided $\E (V( R^N_0 (i )  )   ) < \infty .$

In particular, under our assumptions, 
\begin{equation}\label{eq:boundrt}
 \sup_{ t \geq 0  } \E ( R_t^{N } (i )  ) \le \kappa_R  :=  \E ( R_0^{N  } (1)  )   + \frac{K}{\lambda}  ,
\end{equation}
implying that 
\begin{equation}\label{eq:boundut}
\sup_{t  \geq 0  } \E ( U_t^{N } (i )  ) \le  \kappa_U := \E ( U_0^{N } (1)  ) +   \frac{\alpha K}{\beta} \kappa_R .
\end{equation}

\subsection{Tightness in Skorokhod space}
We consider a probability distribution $\eta_0$
on $\R^2_+$ such that $ \int_{\R_+^2 } (u + r ) \eta_0 (du, dr ) < \infty, $ and for each $N\geq 1$,
the unique solution $(U^N_t, R^N_t)_{t\geq 0}$ to \eqref{eq:dyn} starting from
some i.i.d. $\eta_0$-distributed initial conditions $(U^{N}_0 (i ) , R^{N }_0 (i ) )$. We want to show that the sequence of processes $(U^{N}_t (i ) , R^{N }_t (i ) )_{t\geq 0}$ is tight in $D(\R_+ , \R_+^2),$
for any $ i \geq 1.$ Here, the set $\D(\R_+, \R_+^2 )$ of c\`adl\`ag functions on $\R_+$ taking values in $\R_+^2 $ 
is endowed with the topology of the Skorokhod convergence on compact time intervals,
see  \cite{js}.

\begin{prop}\label{prop:6}
Grant Assumption  \ref{ass:1} and let $(U^{N}_0 (i ) , R^{N }_0 (i ) )_{1 \le i \le N} $ be i.i.d.\ satisfying that  $ \E ( U^{N}_0 (i ) ) < \infty $ and $ \E( R^{N }_0 (i ) ) < \infty . $ \\
(i) The sequence of processes $(U^{N}_t (1) , R^{N }_t (1) )_{t\geq 0}$ is tight in $D(\R_+ , \R_+^2)$.
\\
(ii) The sequence of empirical measures $\hat \mu_N=N^{-1}\sum_{i=1}^N \delta_{(U^{N}_t (i ), R^{N}_t (i ) )_{t\geq 0}}$
is tight in ${\mathcal P} (\D(\R_+, \R_+^2 ))$.
\end{prop}

\begin{proof}
Point (ii) follows from point (i) and the exchangeability 
of the system, see  \cite[Proposition 2.2-(ii)]{s}. We thus
only prove (i). 
To show that the family $((U^{N}_t (1) , R^{N }_t (1) ))_{t\geq 0})_{N\geq 1}$ is tight in  $\D(\R_+, \R_+^2)$, 
we use the criterion of Aldous, see  \cite[Theorem 4.5 page 356]{js}. It is sufficient to prove that

(a) for all $ T> 0$, all $\varepsilon >0$,
$ \lim_{ \delta \downarrow 0} \limsup_{N \to \infty } \sup_{ (S,S') \in A_{\delta,T}} 
\P ( |U_{S'}^{N } (1)  - U_S^{N }(1)  | + |R_{S'}^{N } (1)  - R_S^{N  } (1)  | > \varepsilon ) = 0$,
where $A_{\delta,T}$ is the set of all pairs of stopping times $(S,S')$ such that
$0\leq S \leq S'\leq S+\delta\leq T$ a.s.,

(b) for all $ T> 0$, $\lim_{ L \uparrow \infty } \sup_N 
\P ( \sup_{ t \in [0, T ] } ( |U_t^{N } (1) |  + | R_t^{N} (1)  | ) \geq L ) = 0$.

To check (a), consider $(S,S')\in A_{\delta,T}$ and write
\begin{multline*}
\E |U_{S'}^{N } (1)  - U_S^{N } (1) | \le \beta  \int_S^{S'} \E |U_s^{N } (1) | ds +
   \frac{\alpha }{ {N} } \sum_{j=1}^N \int_S^{S'} \E [\varphi ( U_{s }^{N } (j) ) R_s^{N } (j) ]  ds   \\
\le \beta \int_S^{S'} \E |U_s^{N } (1) | ds + \alpha K  \kappa_R \delta 
\le \beta \kappa_U  \delta + \alpha K  \kappa_R  \delta  \to 0
\end{multline*}
as $ \delta \to 0.$ The expression $\E |R_{S'}^{N } (1)  - R_S^{N } (1) |$ is treated analogously. Moreover, (b) immediately follows from the a priori bounds \eqref{eq:boundrt} and \eqref{eq:boundut}, since $\kappa_R $ and $\kappa_U $ do not depend on $N.$ 
\end{proof}

At this point, one usually concludes the proof that the sequence of processes $(U_t^{N } (i), R^{N }_t(i) )_{t\geq 0}$ converges weakly to $ (U_t, R_t)_{ t \geq 0} $ in  $D(\R_+ , \R_+^2)$ by showing that any 
possible limit point of the sequence is necessarily solution of the limit equation. Classically, this is shown by proving that any limit law must be solution of the associated martingale problem. Uniqueness of the 
solution of the martingale problem then implies the desired convergence.  

In our specific situation however, we are able to identify any possible limit thanks to a coupling argument that we shall present in the next subsection. This coupling argument has another advantage. It enables us to give
a precise rate of convergence.

\subsection{A coupling approach}
We propose a coupling approach, which is inspired by the ideas presented in Sznitman (1991) \cite{s}. The non-Lipschitz term $  \E[\varphi (U_s) R_s ]ds $ appearing in the limit system \eqref{eq:dynlimit} 
demands however to adapt this approach to the present situation. 
Throughout this section we work under Assumption \ref{ass:2} implying that $U^{N }_0 (i)  = u_0 $ for all $ N $ and $1 \le i \le N.$ As a consequence, coming back to \eqref{eq:dyn}, we see that 
$$ U_t^{N } (i)  = U_t^{N } (1)  =: U_t^N $$ 
for all $ 1 \le i \le N, $ that is, the membrane potential processes of all neurons within the system are all equal, and only the values of the calcium concentrations differ. We can therefore rephrase \eqref{eq:dyn} as 
\begin{eqnarray}\label{eq:dynbis}
U^{N}_t &= & u_0  - \beta \int_0^t  U^{N}_s   ds + \frac{\alpha}{N}\sum_{ j =1 }^N \int_0^t  \int_0^\infty  R^{N}_{s-} (j)  1_{ \{ z \le  \varphi ( U^{N}_{s-}) \}} \bM^j (ds,  dz) ,   \\
R^{N }_t (i)  &=& R^{N }_0 (i)  - \lambda \int_0^t  R^{N}_s (i)    ds  +   \int_0^t \int_0^\infty 
  1_{ \{ z \le  \varphi  ( U^{N}_{s-}) \}} \bM^i (ds, dz) , 1 \le i \le N . 
\nonumber
\end{eqnarray}   
In order to control the speed of convergence to the limit system, we now first replace  
 \eqref{eq:dynbis} by an approximating system which is given as follows. 
\begin{eqnarray}\label{eq:dynapprox}
\tilde U^{N}_t &= &u_0   - \beta \int_0^t  \tilde U^{N}_s   ds +  \alpha \int_0^t  \varphi( \tilde U_s^{N } )  \left(  \frac{1}{N} \sum_{j=1}^N \tilde R_s^{N } (j) \right) ds ,   \\
\tilde R^{N }_t (i)  &=& R^{N }_0 (i)  - \lambda \int_0^t  \tilde R^{N}_s (i)    ds  +   \int_0^t \int_0^\infty 
  1_{ \{ z \le  \varphi  ( \tilde U^{N}_{s-}) \}} \bM^i (ds, dz) ,
\nonumber
\end{eqnarray}  
where $ \bM^i (ds, dz) $ is the Poisson random measure driving the dynamics of $  R^{N } (i) .$  

Our aim is to show that \eqref{eq:dynapprox} is close to the original system \eqref{eq:dynbis}. To do so, we
introduce the distance

\begin{equation}
\Delta_t^{N} = \E \left( \left[ | \tilde U_t^{N } - U_t^{N} | + |  \tilde R_t^{N } (1) - R_t^{N } (1)  | \right]  \right) .
\end{equation}

\begin{theo}\label{theo:coupling}
Grant Assumptions \ref{ass:1} and \ref{ass:2}. Fix $T > 1 .$ Then for all $ t \le T ,$
\begin{equation}
\Delta_t^{N}  \le C_T N^{-1/2} .
\end{equation}
\end{theo}

\begin{proof}
We work on the fixed time interval $ [0, T ] $ and we suppose w.l.o.g. that $ K T \geq 1 .$  
Rewrite first the equation of $U_t^{N} $ in the following way. For $ \tilde \bM^j (ds, dz ) := \bM (ds, dz ) - ds dz ,$ we have 
\begin{eqnarray*}
U^{N}_t &=&  u_0  - \beta \int_0^t  U^{N}_s   ds + \frac{\alpha}{N}\sum_{ j=1  }^N \int_0^t  \int_0^\infty  R^{N }_{s-} (j)  1_{ \{ z \le  \varphi ( U^{N}_{s-}) \}} \tilde \bM^j (ds,  dz) 
\\
&& \quad \quad \quad + \alpha \int_0^t  \varphi ( U_s^{N} ) (\frac{1}{N}  \sum_{j=1  }^N  R_s^{N } (j)  ) ds \\
&=:& u_0  - \beta \int_0^t  U^{N}_s   ds  + \alpha \int_0^t  \varphi ( U_s^{N} )  (\frac{1}{N}  \sum_{j =1 }^N R_s^{N } (j)  )ds  +  M_t^{N, 1}  ,
\end{eqnarray*}
where 
\begin{equation}\label{eq:m1}
M_t^{N, 1 }  :=   \frac{\alpha}{N}\sum_{ j =1}^N \int_0^t  \int_0^\infty  R^{N }_{s-} (j)  1_{ \{ z \le  \varphi ( U^{N}_{s-}) \}} \tilde \bM^j (ds,  dz) .
\end{equation}

Therefore, writing 
\begin{equation}\label{eq:barrn}
\bar R_t^N :=  \frac{1}{N} \sum_{j =1 }^N \tilde  R_t^{N } (j)  ,
\end{equation}
we have for a constant $C$ that might change from one occurrence to another
\begin{eqnarray}\label{eq:tobeusedlaterone}
 | \tilde U_t^{N  } - U_t^{N } | 
    &\le & \beta \int_0^t | \tilde U_s^{N } - U_s^{N} |  ds 
  +  \alpha \int_0^t \frac{1}{N} \sum_{j  } | \varphi ( U_s^{N} ) R_s^{N} (j)    - \varphi ( \tilde U_s^{N} ) \tilde R_s^{N } (j)  |ds 
  +  |M_t^{N, 1 } |  \nonumber \\
& \le& \beta \int_0^t | \tilde U_s^{N  } - U_s^{N } |  ds   +  \alpha \int_0^t |\varphi ( U_s^{N} )    - \varphi ( \tilde U_s^{N} )|  \bar R_s^{N }  ds \nonumber \\
&& +   C   \int_0^t \frac{1}{N} \sum_j  |R_s^{N } (j)  - \tilde R_s^{N}  (j) | ds +  |M_t^{N  , 1} | ,
\end{eqnarray}  
where we have used the boundedness of $ \varphi  .$ 

In the above expression, we have to control the term $ \alpha \int_0^t |\varphi ( U_s^{N} )    - \varphi ( \tilde U_s^{N} )|  \bar R_s^{N }  ds .$ For that sake, fix a constant  $a  \geq  \sup_{ s  \geq 0  } \E \check  R_s^N (1)  + 1.  $ Notice that by \eqref{eq:tobeusedlater}, applied with $ V(x) = x^2,$  we can choose $a$ such that it does not depend on $T. $ Then, by the Lipschitz continuity of $ \varphi, $ 
\begin{multline*}
 \alpha \int_0^t |\varphi ( U_s^{N} )    - \varphi ( \tilde U_s^{N} )|  \bar R_s^{N }  ds \\
\le  a  \alpha L_\varphi \int_0^t | U_s^{N}- \tilde U_s^{N} |  ds +  \alpha \int_0^t |\varphi ( U_s^{N} )    - \varphi ( \tilde U_s^{N} )|  \bar R_s^{N }  1_{\{    \bar R_s^N \geq a \}} ds .
\end{multline*}
Since $ \varphi $ is bounded, the last expression in the above term is bounded by 
$$ \alpha \| \varphi\|_\infty \int_0^t  \bar R_s^{N }  1_{\{    \bar R_s^N \geq a \}} ds  .$$
We use H\"older's inequality to obtain 
$$  \E (  \bar R_s^{N }  1_{\{ \bar R_s^N \geq a \}})  \le \left(\E [ (\bar R_s^{N })^2] \right)^{1/2} \left(\P (    \bar R_s^N \geq a )\right)^{1/2} .$$  
\eqref{eq:tobeusedlater} applied with $ V(x) = x^2 $ and Chebyshev's inequality yield, by independence of  the processes $\check R^N(i), 1 \le i \le N, $ 
$$ \P ( \bar R_s^N > a ) \le \P (\sum_{i=1}^N  \check R_s^N (i) > N a ) \le   \P (  \sum_{i=1}^N  \check R_s^N (i) -  \E \check R_s^N (i)  >   N  ) \le \frac1N  Var (\check R_s^N (1 ) ) \le \frac{C }{ N} ,$$
since $ a  \geq   \E \check R_s^N (i ) + 1 .$ 
Therefore
$$\alpha \| \varphi \|_\infty \int_0^t   \E [ \bar R_s^{N }  1_{\{    \bar R_s^N \geq a \}}] ds   \le  C t N^{-1/2}. $$ 
We conclude that for all $t \le T,$ 
\begin{multline}\label{eq:upperboundu}
\E \left(  | \tilde U_t^{N  } - U_t^{N } |  \right)  \le \beta \int_0^t \E | \tilde U_s^{N  } - U_s^{N} |   ds     +   a \alpha L_\varphi    \int_0^t  \E  |  U_s^{N}    -   \tilde U_s^{N} | ds \\
+ C t N^{-/1/2} +  C   \int_0^t  \E  | R_s^{N } (1)  - \tilde R_s^{N } (1)  |  ds+  \E | M_t^{N, 1} |\\
 \le  C  \int_0^t \Delta_s^{N } ds  +  \frac{C T}{ \sqrt{N} }  + \E | M_t^{N, 1 } |,
\end{multline}
where we have used the exchangeability of the system.

To deal with the martingale part, write
$$M_t^{N, 1 } = \frac{\alpha}{N}  \sum_{j=1 }^N M_t^{N,1, j} , \;\mbox{ for } \;    M_t^{N, 1,j} = \int_0^t  \int_0^\infty  R^{N }_{s-} (j) 1_{ \{ z \le  \varphi ( U^{N}_{s-}) \}} \tilde \bM^j (ds,  dz) .$$
Notice that $M_t^{N,1, j } $ and $M_t^{N, 1, k } $ almost surely never jump together for $ j \neq k , $ hence, using once more \eqref{eq:tobeusedlater} with $ V(x) = x^2,$ 
$$
 \E (M_t^{N,1})^2 = \frac{ \alpha^2}{N^2} \sum_{j = 1 }^N \E ( M_t^{N,1,  j })^2 \le C N^{-1} \int_0^t \E \left( (R_s^{N} (1) )^2 \varphi ( U_s^{N} ) \right) ds 
\le  C  \frac{t}{N}   .  
$$
As a consequence, for all $t \le T,$ 
$$  \E | M_t^{N, 1} | \le  C T^{1/2}   N^{- 1/2 } .$$ 
Finally, using once again the Lipschitz continuity of $ \varphi, $ 
$$
\E \left( | \tilde R_t^{N } (1) - R_t^{N } (1)  | \right)  \le \lambda  \int_0^t \E |  \tilde R_s^{N } (1) - R_s^{N }(1)  |  ds + C   \int_0^t \E  | \tilde U_s^{N  } - U_s^{N } |  ds ,
$$
whence
$$ \Delta_t^{N}  \le C ( \alpha ,  K, g_0  )  T      N^{-1/2}   + C   \int_0^t \Delta_s^{N, a} ds  ,$$
implying that 
$$ \Delta_t^{N}  \le   \; C T  N^{-1/2} e^{Ct} $$
for all $ t \le T,$ which concludes the proof.
\end{proof}
 
We are now going to control the distance between the approximating system \eqref{eq:dynapprox} and the limit system \eqref{eq:dynlimit}. Recall that 
$$ \bar R_t^N:= \frac1N \sum_{i=1}^N \tilde R_t^{N } (i)  .$$
Then
 $$ d \bar R_t^N = - \beta \bar R_t^N dt + \int_0^\infty 
  1_{ \{ z \le  \varphi  ( \tilde U^{N}_{t-}) \}} \overline \bM^N ( dt, dz ) ,\;  \mbox{ where  } \, \;  
 \overline \bM^N = \frac1N \sum_{i=1}^N \bM^i .$$
Compensating each Poisson random measure, this yields
$$ d \bar R_t^N = - \beta \bar R_t^N dt  + \varphi ( \tilde U^N_t ) dt +  d M_t^{N , 2} , $$
where
\begin{equation}\label{eq:m2}
M_t^{N, 2}  = \frac{1}{{N}} \sum_{i=1}^N  \int_0^t \int_0^\infty 
  1_{ \{ z \le  \varphi  ( \tilde U^{N}_{s-}) \}} \tilde \bM^i (ds, dz) 
\end{equation}  
is a square integrable martingale. We can thus rewrite \eqref{eq:dynapprox} as 
\begin{eqnarray}\label{eq:dynapproxter}
\tilde U^{N}_t &= & u_0   - \beta \int_0^t  \tilde U^{N}_s   ds +  \alpha \int_0^t  \varphi( \tilde U_s^{N } ) \bar R_s^N ds ,   \\
\bar R_t^N &=& \bar R^{N }_0 - \lambda \int_0^t \bar R_s^N  ds  +   \int_0^t \varphi ( \tilde U^N_s) ds +    M_t^{N, 2 }.
\end{eqnarray}
Moreover, since $ U_0 = u_0  $ is deterministic, writing $ r_t := \E R_t $ and recalling that $u_t$ is deterministic,  the dynamics of the limit equation boils down to 
\begin{equation}\label{eq:system}
\left\{ 
\begin{array}{lcl}
d u_t &=&  - \beta u_t dt +  \alpha \varphi ( u_t) r_t dt\\
d r_t &=& - \lambda r_t dt + \varphi (u_t ) dt  
\end{array}
\right\} .  
\end{equation}
We obtain 

\begin{theo}\label{theo:couplingbis}
Grant the assumptions of Theorem \ref{theo:coupling}. Fix $T > 1 .$ Then for all $ t \le T $
$$ \E \left[  | u_t - \tilde U^{N}_t | +  | r_t - \bar R_t^N| \right] \le C_T  N^{-1/2} $$
and consequently for the original particle system, 
$$ \E [  | u_t -  U^{N}_t | ] +  \E \left[ \Big | \frac1N \sum_{j=1}^N   R_t^{N } (j)   - r_t \Big| \right] \le C_T  N^{-1/2} .$$
\end{theo} 

\begin{proof}
\eqref{eq:system} together with \eqref{eq:dynapproxter} implies 
\begin{multline}\label{eq:devubis}
 | u_t - \tilde U^{N}_t | \le \beta \int_0^t  | u_s- \tilde U^{N}_s | ds+ \alpha \int_0^t | \varphi( \tilde U_s^{N } )\bar R_s^N - \varphi ( u_s) r_s | ds  \\
 \le  \beta \int_0^t  | u_s- \tilde U^{N}_s |  ds + C \int_0^t | \ \tilde U_s^{N }- u_s|  r_s   ds +   \alpha  K \int_0^t | r_s - \bar R_s^N | ds .
\end{multline}
We pass to expectation and use the fact that $r_s$ is deterministic and bounded by $ r_s \le r_0 + \frac{K}{  \lambda } , $ 
to deduce from this that 
$$  \E  | u_t - \tilde U^{N}_t | \le C \int_0^t  \left( \E  | u_s- \tilde U^{N}_s | + \E  | r_s -\bar R_s^N | \right) ds .$$
Moreover, 
\begin{equation}\label{eq:devrbis}
 | r_t - \bar R^{N}_t | \le | r_0 - \bar R^N_0 | + \lambda \int_0^t  | r_s- \bar R^{N}_s | ds+L_\varphi \int_0^t |  \tilde U_s^{N } -  u_s|   ds +  | M_t^{N, 2} |  .
\end{equation}
We use that $M_t^{N, 2}$ has quadratic variation
$$ <M^{N, 2} >_t = \frac1N \int_0^t \varphi ( \tilde U^N_s ) ds \le \frac{Kt}{N} $$
to obtain
\begin{equation}\label{eq:also}
\E | r_t - \bar R^{N}_t | \le \E  | r_0 - \bar R^N_0 | + \lambda \int_0^t \E  | r_s- \bar R^{N}_s | ds+C  \int_0^t \E |  \tilde U_s^{N } -  u_s|   ds + C_T \frac{1}{\sqrt{N}}  , 
\end{equation}
implying the assertion, since $\E  | r_0 - \bar R^N_0 | \le C N^{-1/2} .$ 
\end{proof}

Theorems \ref{theo:coupling} and \ref{theo:couplingbis} enable us finally to conclude that 

\begin{theo}\label{theo:conv}
Grant the assumptions of Theorem \ref{theo:coupling}. Then we have that for all $ i \geq 1, $ 
the sequence of processes $(U^{N}_t (i) , R^{N }_t (i) )_{t\geq 0}$ converges weakly to $ (u_t, R_t)_{ t \geq 0} $ in  $D(\R_+ , \R_+^2)$.
\end{theo}

\begin{proof}
Take any weakly convergent subsequence of $ (U^{N}_t (i) , R^{N }_t (i) ) $ and call $Z=  ( V, S ) $ its weak limit. We suppose that $ (V, S) $ is defined on a filtered probability space $ ( \Omega', {\mathcal A}', ( {\mathcal F_t}')_{t \geq 0}, P' ) , $ where 
$$  {\mathcal F_t}' = \bigcap_{ s > t }  {\mathcal F_s}^0, {\mathcal F_s}^0 = \sigma ( V_t , S_t, t \le s  ) .$$  Theorem \ref{theo:coupling} and \ref{theo:couplingbis} imply that $ V = u $ almost surely (since $u$ is deterministic). Moreover it is straightforward to show that the limit $(u, S) $ must be solution of the following martingale problem. For  any smooth and bounded test function $ \psi ,$ any  $ s_1 < s_2 < \ldots < s_k \le s < t, $ for  continuous and bounded test functions $ \psi_i ,$ we have
\begin{equation} 
\E \left[ \left( \psi (S_t) -\psi ( S_s) - \lambda \int_s^t \psi' ( S_v) S_v  dv - \int_s^t  \varphi ( u_v ) [ \psi( S_v+ 1 ) - \psi( S_v) ]  dv  \right)  \prod_{i=1}^k \psi_i ( Z_{s_i} ) \right] = 0 .
\end{equation}
By \cite[Theorem II.2.42, page 86]{js}, and using the right-continuity of $S,$ this implies that  $S$ is a $( P', ( {\mathcal F_t}')_{t \geq 0})-$semi-martingale with characteristics $ B = - \lambda \int_0^\cdot  S_s ds , $ $ \nu (ds , dx ) = \varphi (u_s) ds \delta_1 (dx) , $ $C_t=0  .$ Moreover, \cite[Theorem III.2.26, page 157]{js} implies that there exists a Poisson random measure $ \pi $ defined on $ ( \Omega', {\mathcal A}', ( {\mathcal F_t}')_{t \geq 0}, P' ) , $ such that $S$ is solution of 
$$ S_t = S_0 - \lambda \int_0^t S_v dv +  \int_0^t \int_0^\infty 1_{\{ z \le \varphi ( u_s ) \}} \pi (ds, dz ) ,$$
where $ S_0 $ is $g_0-$distributed. In other words, $ S \stackrel{\mathcal L}{=} R.$ Hence any weak limit has the same law, implying the weak convergence of $ (U^{N}_t (i) , R^{N }_t (i) ) .$   
\end{proof}

\section{Stationary solutions of the limit equation}
For smooth spiking rate functions $ \varphi , $ by Theorem \ref{prop:dieout},  the finite size system has only one invariant state corresponding to extinction of the system. The limit system can however display several invariant states as we show now, including persistent behavior where the spiking activity of the system survives forever. 

Recall that passing to expectation and writing $ u_t = \E U_t , r_t = \E R_t, $ we have reduced the limit system  to 
\begin{equation}\label{eq:systembis}
\left\{ 
\begin{array}{lcl}
d u_t &=&  - \beta u_t dt +  \alpha \varphi ( u_t) r_t dt\\
d r_t &=& - \lambda r_t dt + \varphi (u_t ) dt  
\end{array}
\right\} .  
\end{equation}
Any stationary solution $ (u^* , r^* ) $ of \eqref{eq:systembis} must satisfy 
\begin{equation}
 \lambda r^* = \varphi ( u^* ) \; \mbox{ and } \;  \beta u^* = \alpha \varphi ( u^* ) r^* =\frac{ \alpha}{ \lambda} \varphi^2 (u^* )  ,
\end{equation}
implying that
\begin{equation}\label{eq:fixpoint}
u^* = \frac{ \alpha}{ \beta  \lambda} \varphi^2 (u^* )  .
\end{equation}
Of course, $ ( 0, 0 )$ is always a stationary solution since $ \varphi( 0) = 0 . $ However, for suitable choices of 
$$ \kappa :=  \frac{ \alpha}{ \beta  \lambda} $$ 
($\kappa$ is homogeneous to a potential times a time squared) and of the form of $ \varphi, $ also other stationary solutions appear, some of them  being attracting. Let us come back to the example already presented in Section 2.

\begin{ex}\label{ex:ph}
We consider spiking rate functions of sigmoid type. They are defined in terms of a parameter  $a  > 1  $ satisfying that  
$$ \frac{4a }{1 + e^{ a} } < 1 .$$ 
Let   
$$ \varphi ( x) = \frac{4a}{ 1 + e^{-(x -a)}} - \frac{4a}{ 1 + e^{ a}} , x \geq 0 .$$
Then $\varphi ( 0) = 0, $ $0 < \varphi ' (0) < 1 ,$ $\varphi'' > 0 $ on $ [ 0, a [ $ and $\varphi '' < 0 $ on $] a, \infty [ ,$ thus $ a$ is the inflexion point of $ \varphi .$ 
We have that  
$$ \varphi (a ) =  4 a \left( \frac12 - \frac{1}{ 1 + e^{ a}} \right) \geq 4 a \left( \frac12 - \frac14 \right) = a , $$
since $ \frac{4a }{1 + e^{ a} } < 1 $ and $ a > 1 $ imply that $  \frac{1 }{1 + e^{ a} } < \frac14. $ 
Notice that $ \varphi^2  (0) = 0 ,$ $ (\varphi^2 )' (0 ) = 0 $ and $\varphi^2  (a) = a^2 > a $ (recall that  $ a > 1$). This implies that for some $ x_1 \in ] 0, a [  $ we have  $ \varphi^2  ( x_1 ) = x_1 .$ Finally, boundedness of $\varphi^2  $ implies the
existence of a second point $ x_2 > a $ with $ \varphi^2  ( x_2 ) = x_2 .$   
\end{ex}

Our Assumption \ref{ass:psi} implies the existence of at least two non trivial solutions of  $\kappa \varphi^2  ( x) = x $ in $] 0, \infty [  , $ for all $ \kappa \geq D.$ This means that two non-trivial stationary solutions for \eqref{eq:systembis} exist. Let $  u^{max} > 0  $ be the maximal solution of  $\kappa \varphi^2  ( x) = x $ in $] 0, \infty [   $ and  write  $ r^{max} := \frac{1}{\lambda } \varphi ( u^{max} ) .$ The point $ ( u^{max}, r^{max} ) $ is locally attracting as shows the following proposition.

\begin{prop}\label{lem:attr}
Grant Assumption \ref{ass:psi} and suppose that $\kappa \geq D. $ Then  $(0,0)$ and $  (u^{max}, r^{max}  )$ are locally attracting equilibria of \eqref{eq:systembis}. 
\end{prop}

\begin{figure} 
\begin{center}
   \includegraphics[width=17cm]{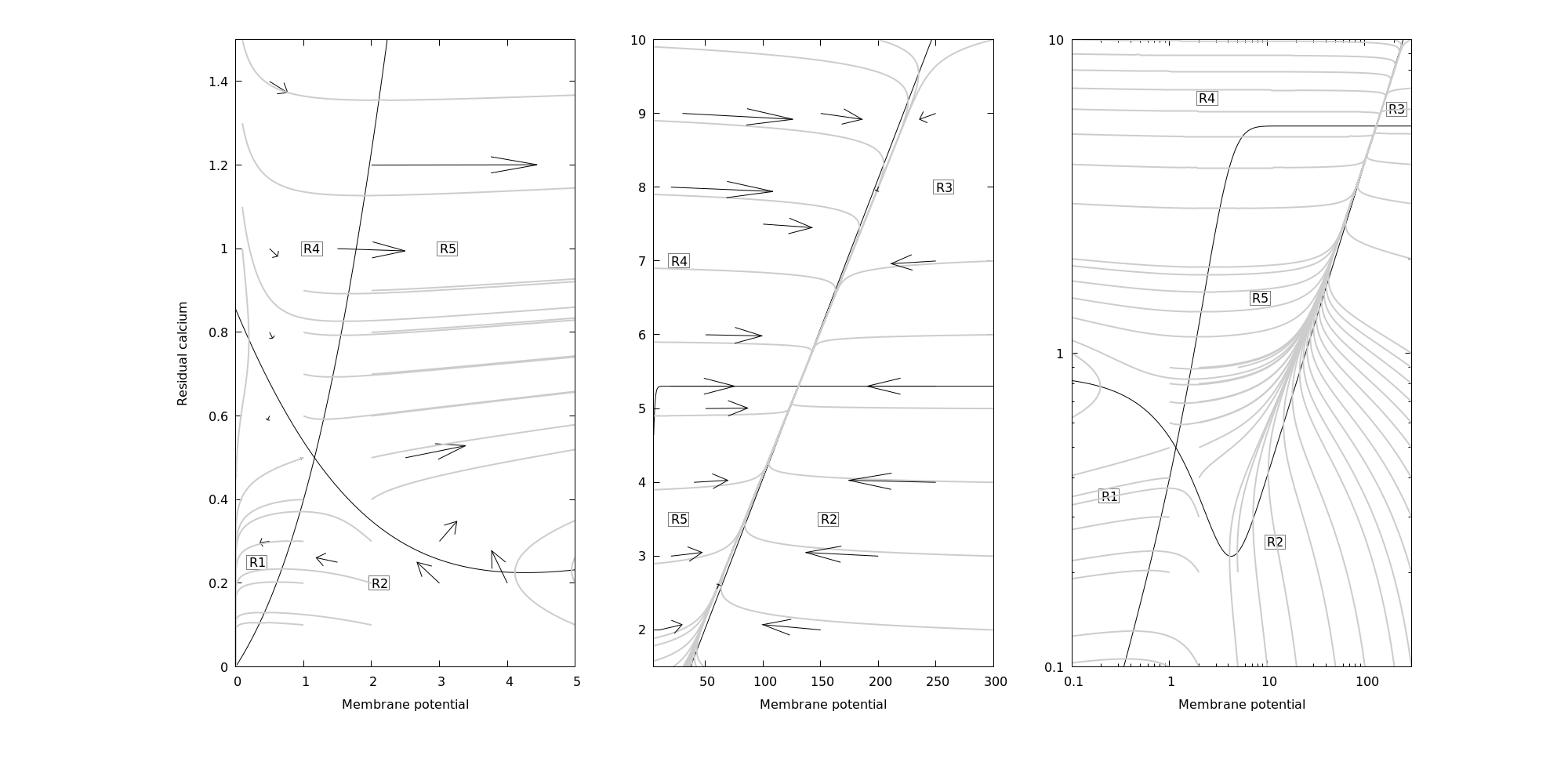} 
\end{center}
\caption{Phase portrait of the dynamical system given in \eqref{eq:systembis} with $\varphi $ as in Example \ref{ex:ph} and $\alpha=107.78$; $\beta=50$; $\lambda=2.16$; $a=3.$ Left : Plot showing the unstable equilibrium $ (x_1, r_1), r_1 = \frac{1}{\lambda } \varphi ( x_1) .$ The black curve starting from $(0, 0)$  is the null-cline of $r, $ the other black curve the null-cline of $u.$ Trajectories starting in $R_1$ are attracted to $ (0,0). $ Middle : Plot showing the stable equilibrium $ (x_2, r_2 ), $ the upper black curve is the null-cline of $r.$ The arrows represent the tangent vectors of the flow, their length correspond to $ 1/100$ of their true length. Right : Plot on a log-log scale showing both equilibria.}
\end{figure}

\begin{proof}
Write $ \Phi (u ) = u / \varphi (u ) .$ Then $ \{ r = \frac{1}{\lambda}  \varphi (u ) \} $ is the null-cline of $r_t $ and $ \{ r = \frac{\beta}{\alpha}  \Phi (u ), u > 0 \} \cup \{ r= u = 0 \} $ the null-cline of $ u_t .$ We suppose w.l.o.g. that $\kappa \varphi^2  ( x) = x $ possesses exactly two non trivial solutions $ 0 < x_1 < x_2 = u^{max} $ in $ ] 0, \infty [.$ Then
we have five regions 
$$ R_1 = \{ ( u, r) : u \le x_1 , \frac{1}{\lambda} \varphi ( u ) \le r \le \frac{\beta}{\alpha} \Phi (u ) \} , R_2 = \{ ( u, r ) : r \le \frac{1}{\lambda}  \varphi (u ) \wedge  \frac{\beta}{\alpha} \Phi (u ) \} ,$$
$$ R_3 = \{ ( u, r) : u \geq x_2 , \frac{1}{\lambda}  \varphi ( u ) \le r \le  \frac{\beta}{\alpha} \Phi (u ) \},  R_4 = \{ (u, r ) : r \geq \frac{1}{\lambda}  \varphi (u ) \vee  \frac{\beta}{\alpha} \Phi (u )\}$$
and finally
$$ R_5 = \{ (u, r ) : x_1 \le u \le x_2,  \frac{\beta}{\alpha} \Phi ( u ) \le r \le \frac{1}{\lambda}  \varphi (u ) \} ,$$
see Figure 3.  
On $R_1 \cup R_3,$ both $ u_t $ and $r_t $ decrease. On $ R_2, $ only $r_t$ increases while $ u_t$ decreases. Thus, for sufficiently small $ u_0, $ trajectories starting within $R_2 $ will enter $ R_1 $ after some time and then be attracted to $ (0, 0) .$ The other trajectories starting within $R_2$ will either enter $ R_5 $ or $R_3.$ In both cases, they finish being attracted to $ (u^{max}, r^{max}  ).$ The same argument applies to trajectories starting in  $ R_4, $ where $ u_t$ increases, while $r_t$ decreases. 
\end{proof}

Let us finally shortly discuss the role of $\kappa $ on the number of stationary limit states. 

\begin{prop}
Grant Assumption \ref{ass:psi}. There exists $ \kappa_c \le D $ such that for all $ \kappa < \kappa_c,$ no non-trivial solution to \eqref{eq:fixpoint} exist, and for all $ \kappa > \kappa_c , $ at least two non-trivial solutions exist.
\end{prop}

\section{Deviation inequalities and proof of Theorem \ref{theo:4}}
We are now able to conclude the paper with the proof of Theorem \ref{theo:4}. In what follows we shall use the martingales $ M^{N, 1} $ and $M^{N, 2 } $ which have been defined in \eqref{eq:m1} and \eqref{eq:m2}. We also introduce
$$ M_t^{N, 3 } = \frac1N \sum_{i=1}^N \int_0^t \int_{\R_+} | 1_{ \{ z \le \varphi ( U^N_{s-} ) \} } -  1_{ \{ z \le \varphi ( \tilde U^N_{s-} ) \} } | \tilde \bM^i (ds, dz ) .$$
Finally we use the definition of $\bar R_t^N, $ given in \eqref{eq:barrn} above. 

The first item of Theorem \ref{theo:4} will then be a consequence of 
\begin{theo}
Fix $T > 1 $  and grant the assumptions of Theorem \ref{theo:coupling}. Suppose moreover that $ \int_0^\infty e^{ u r} g_0 ( dr ) < \infty $ for some $ u > 0 .$  Then for any $ R > 0 , $ there exists a constant $C_R$ depending only on the parameters of the model and on  $R,$ such that 
\begin{multline}\label{eq:devontun}
\sup_{s \le T} \left( | U_s^N - u_s |  + | \frac1N \sum_{i=1}^N R_s^{N } (i)  - r_s | \right) \le\\
 (C_R)^{T+1}  \left(  \int_0^T  1_{ \{ \bar R_{s}^N > R  \} } \bar R_{s}^N ds + | \bar R_0^N - r_0 | + \sup_{s \le T} | M_s^{N, 1} |  + \sup_{s \le T} | M_s^{N, 2 }| + \sup_{s \le T } | M_s^{N, 3} |  \right) .
\end{multline}
Consequently for all $ \varepsilon > 0 $ there exists $ N_0 = N_0 ( T, \varepsilon ), $ such that for all $ N \geq N_0,$ we obtain the following deviation bound
\begin{equation}
 \P \left( \sup_{s \le T} \left( | U_s^N - u_s |  + | \frac1N \sum_{i=1}^N R_s^{N } (i)  - r_s | \right)\geq  N^{-1/5}  \varepsilon \right) \le \\
C_T  e^{ - c  \varepsilon^2  \sqrt{N} }.
\end{equation} 
\end{theo}

\begin{proof} In the sequel we shall suppose w.l.o.g. that $ KT \geq 1.$\\
{\it Step 1.}
Fix some  $0 < t_0 < T $ and let  $ t > 0 $ such that $t_0 + t \le T.$  Then Equation \eqref{eq:tobeusedlaterone} implies, for any $ R > 0 , $ 
\begin{multline}\label{eq:tobeusedlatertwo}
\sup_{t_0 \le s \le t_0 + t } | U^N_s - \tilde U^N_s | \le | U^N_{t_0} - \tilde U^N_{t_0}  | +   t \cdot \left( \beta  + \alpha  L_\varphi R \right)  \sup_{t_0 \le s \le t_0 +  t } | U^N_s - \tilde U^N_s |  \\
+ \alpha  K  \int_0^T 1_{ \{ \bar R_s^N > R  \}} \bar R_{s}^N  ds + C t \frac1N \sum_{i=1}^N \sup_{ t_0 \le s \le t_0 + t } | R_s^{N } (i)  - \tilde R_s^{N } (i)  | + \sup_{t_0 \le s \le t_0 + t} | M_s^{N, 1 } | .
\end{multline}
Applying the same argument to $R^N  $ and $\tilde R^N, $ we also have that 
\begin{multline*}
 \frac1N \sum_{i=1}^N \sup_{ t_0 \le s \le t_0 +  t } | R_s^{N } (i)  - \tilde R_s^{N } (i)  |  \le \frac1N \sum_{i=1}^N  | R_{t_0}^{N } (i)  - \tilde R_{t_0}^{N } (i)  | +  \lambda t \left( \frac1N \sum_{i=1}^N \sup_{ t_0 \le s \le t_0 + t } | R_s^{N } (i)  - \tilde R_s^{N } (i)  | \right) \\
 + \sup_{t_0 \le s \le t_0 +  t } | M_s^{N, 3 } | + L_\varphi t \sup_{t_0 \le s \le t_0 +  t }  | U^N_s - \tilde U^N_s |  .
\end{multline*} 
Choosing $s_1 $ sufficiently small such that $ \lambda s_1 \le \frac12 $ and also $ 2 L_\varphi s_1 \le 1,$  we deduce from this that for all $ t \le s_1,$ 
\begin{multline}\label{eq:devr}
 \frac1N \sum_{i=1}^N \sup_{ t_0 \le s \le t_0 + t } | R_s^{N } (i)  - \tilde R_s^{N } (i)  |  \le  \frac2N \sum_{i=1}^N  | R_{t_0}^{N } (i)  - \tilde R_{t_0}^{N } (i)  | + 2  \sup_{t_0 \le s \le t_0 + t } | M_s^{N, 3 } | \\
 +  \sup_{t_0 \le s \le t_0 + t }  | U^N_s - \tilde U^N_s | , 
\end{multline} 
and replacing in \eqref{eq:tobeusedlatertwo}, we obtain
\begin{multline*}
 \sup_{t_0 \le s \le t_0 + t } | U^N_s - \tilde U^N_s |  \le  | U^N_{t_0} - \tilde U^N_{t_0}  | +   \frac{2 Ct}{N} \sum_{i=1}^N  | R_{t_0}^{N } (i)  - \tilde R_{t_0}^{N } (i)  | \\ 
 +  t \cdot \left( \beta  + \alpha  L_\varphi R +  C   \right)  \sup_{t_0 \le s \le t_0 + t } | U^N_s - \tilde U^N_s | \\
 +  \alpha K \int_0^T 1_{ \{ \bar R_s^N > R  \}} \bar R_{s}^N ds 
 + 2 C t   \sup_{s \le t } | M_s^{N, 3 } |   + \sup_{s \le t} | M_s^{N, 1 } | .
\end{multline*}
Consequently we may choose $s_2 \le s_1 $ such that 
$$  s_2  \cdot \left( \beta  + \alpha  L_\varphi  R  +  C  \right) \le \frac12 \mbox{ and also } 2 C s_2 \le 1 .$$
Therefore, for all $ t \le s_2, $ for a suitable constant $C,$ 
\begin{multline*}
 \sup_{t_0 \le s \le t_0 + t } | U^N_s - \tilde U^N_s |  \le  2 | U^N_{t_0} - \tilde U^N_{t_0}  | +   \frac{2}{N} \sum_{i=1}^N  | R_{t_0}^{N } (i)  - \tilde R_{t_0}^{N } (i)  | \\
 +  C \left(   \int_0^T 1_{ \{ \bar R_s^N > R  \}} \bar R_{s}^N ds 
+   \sup_{t_0 \le s \le t_0 + t } | M_s^{N, 3 } |   +  \sup_{t_0 \le s \le t_0 + t } | M_s^{N, 1 } | \right) , 
\end{multline*}
and replacing this once more in  \eqref{eq:devr}, we finally obtain for all $ t_0, t$ such that $ t_0 + t \le T,$ 
\begin{multline}\label{eq:firstest}
\sup_{t_0 \le s \le t_0 + t }| U^N_s - \tilde U^N_s | +  \frac1N \sum_{i=1}^N \sup_{t_0 \le s \le t_0 + t } | R_s^{N } (i)  - \tilde R_s^{N } (i)  |\\
 \le C | U^N_{t_0} - \tilde U^N_{t_0}  | 
+    \frac{C}{N} \sum_{i=1}^N  | R_{t_0}^{N } (i)  - \tilde R_{t_0}^{N } (i)  |  \\
 +  C \left(   \int_0^T 1_{ \{ \bar R_s^N > R  \}} \bar R_{s}^N ds 
+  \sup_{ s \le T } | M_s^{N, 3 } |   + \sup_{ s \le T} | M_s^{N, 1 } | \right) 
\end{multline}
for all $t \le s_2.$ 

{\it Step 2.} We apply the above inequality \eqref{eq:firstest} with  $ t_0 = 0, t_0 = s_2, \ldots, t_0 =  \ell s_2 , $ for $ \ell = [ T/ s_2 ], $ where $[x] $ denotes the integer part of $ x> 0, $ and with $ t = s_2 $ (except in the last case $ t_0 =\ell s_2$ where we take $ t = T - \ell s_2$).  Write for short 
$$Y_T :=  C \left(  \int_0^T 1_{ \{ \bar R_s^N > R  \}} \bar R_{s}^N ds 
+  \sup_{ s \le T } | M_s^{N, 3 } |   + \sup_{ s \le T} | M_s^{N, 1 } | \right) $$
and 
$$ \Delta ( k) :=   \sup_{k s_2 \le s \le (k+1) s_2 \wedge T  }| U^N_s - \tilde U^N_s | +  \frac1N \sum_{i=1}^N  \sup_{k s_2 \le s \le (k+1) s_2 \wedge T  } | R_s^{N } (i)  - \tilde R_s^{N } (i)  | , 0 \le k \le \ell .$$ 
\eqref{eq:firstest} implies that 
$$ \Delta ( k+1) \le  C \Delta (k) + Y_T, \mbox{ for all } k < \ell  $$
(choose $t_0 = (k+1) s_2 $ in \eqref{eq:firstest}).
Since $ U^N_0 = \tilde U^N_0 = u_0 $ and $ R^N_0 (i ) = \tilde R^N_0 (i ), $ for all $ 1 \le i \le N, $ we have moreover that $ \Delta ( 0 ) \le Y_T.$ Therefore, 
$$ \Delta ( \ell ) \le (1 + C + \ldots C^\ell )   Y_T \le C^{\ell +1} Y_T $$
(we suppose w.l.o.g. that $ C \geq 1 $).  
Finally, $ \ell \le T/s_2 $ and $s_2 = s_2 (R) $ not depending on $T$ imply that, for  $ C_R= C^{\frac{1}{s_2 (R)} }$ depending only on the parameters of the model and on $R,$ but not on $T,$ 
\begin{multline}\label{eq:nice}  \sup_{s\le  T  }| U^N_s - \tilde U^N_s | +  \frac1N \sum_{i=1}^N \sup_{ s \le T } | R_s^{N } (i)  - \tilde R_s^{N } (i)  | \le \\
\le  (C_R)^{T+1} \left(  \int_0^T 1_{ \{ \bar R_s^N > R  \}} \bar R_{s}^N ds
+  \sup_{ s \le T } | M_s^{N, 3 } |   + \sup_{ s \le T} | M_s^{N, 1 } | \right) .
\end{multline} 
This is the first part of \eqref{eq:devontun}.

{\it Step 3.} We now turn to the second step in our approximation procedure. It is treated analogously to the two previous steps. We have by \eqref{eq:devubis}, for all $t_0, t $ such that $ t_0 + t \le T,$ 
$$\sup_{t_0 \le s \le t_0 + t } | \tilde U^N_s - u_s | \le   | \tilde U^N_{t_0} - u_{t_0} | + t ( \beta + C(r_0 +\frac{K}{\lambda})  ) \sup_{t_0 \le s \le t_0 + t } | \tilde U^N_s - u_s |  + \alpha K t \sup_{t_0 \le s \le t_0 + t } | r_s - \bar R_s^N | ,$$
where we have used that for all $s \le T , $ $ r_s \le r_0 +\frac{K}{\lambda}  .$ 

Analogously, \eqref{eq:devrbis} implies that
$$ \sup_{t_0 \le s \le t_0 + t } | \bar R^N_s -r_s | \le | \bar R^N_{t_0} - r_{t_0} | + \lambda t \sup_{t_0 \le s \le t_0 + t } | r_s - \bar R^N_s| + L_\varphi t \sup_{t_0 \le s \le t_0 + t } | \tilde U^N_s - u_s | + \sup_{ s \le T}| M_s^{N, 2 } |.$$
Putting the two together and using similar arguments as in the first step, for a choice of $t_1  \le s_2 $ not depending on $T$ and sufficiently small, 
\begin{equation}\label{eq:secondest}
\sup_{t_0 \le s \le t_0 + t } |  \tilde U^N_s -u_s |  +\sup_{t_0 \le s \le t_0 + t } |  \bar R^N_s - r_s |  \le C  | \tilde U^N_{t_0} - u_{t_0} | +C | \bar R^N_{t_0} - r_{t_0} | +  C \sup_{ s \le T} | M_s^{N, 2 } |   .
\end{equation}
Iterating the above inequality for $ t_0 = 0, \ldots, t_0 = k t_1, $ with $ k = [ T/ t_1 ]$  as in Step 2. above implies then 
\begin{equation}\label{eq:thirdest}
\sup_{ s \le T } |  \tilde U^N_s -u_s |  +\sup_{ s \le T } |  \bar R^N_s - r_s |  \le  C^{T+1}   | \bar R^N_{0} - r_{0} | +  C^{T+1}  \sup_{ s \le T} | M_s^{N, 2 } |   .
\end{equation}
Finally, \eqref{eq:thirdest} together with \eqref{eq:nice} imply \eqref{eq:devontun}.

{\it Step 4.} We use a large deviations upper bound to obtain
$$ \P (|\bar  R_0^N -r_0|  \geq \varepsilon  N^{-1/4}  ) \le e^{- c \sqrt{N}  } ,$$
for  $ \varepsilon >  0,$ due to our assumptions on the law $ g_0 $ of $ R_0^{N } (1) .$

Moreover, we use similar arguments as those in the proof of Theorem \ref{theo:coupling} to deduce that for any $ \eta > 0, $
$$
\P (\int_0^T 1_{ \{\bar R_s^N > R  \}}  \bar R_{s}^N ds    \geq \eta  ) \le \frac{C}{\eta}  \int_0^T \sqrt{ \P (\bar R_s^N > R )} ds  .
$$
We wish to use again a large deviations upper bound to control $\P (\bar R_s^N > R ).$ Since $ \int_0^\infty e^{u r } g_0 (dr) < \infty , $ applying \eqref{eq:tobeusedlater} with $ V(x) = e^{ux} $ gives 
\begin{equation}\label{eq:expboundrt}
 \sup_{ t \geq 0 } \E e^{ u \check R^N_t (i) } < \infty .
\end{equation} 
Therefore we may  choose $ R $ such that $ \log \left(\sup_{t \geq 0} \E ( e^{ u \check R^N_t (i )} )\right) - u R := - \chi < 0 .$ Then
$$
\P (\bar R_s^N > R ) \le  \P ( \sum_{i=1}^N  \check R_s^N (i )  > N  R ) \le e^{- N u R } \left( \E ( e^{ u \check R^N_s (i )})  \right)^N \le e^{ - \chi N } , 
$$
implying that 
$$ \P (\int_0^T 1_{ \{\bar R_s^N > R  \}}  \bar R_{s}^N ds    \geq \eta  ) \le \frac{C}{ \eta } T e^{- \frac12 \chi N}.$$

To deal with the square integrable martingale terms $ M^{N, i }, 1 \le i \le 3,  $ we rely on the  Bernstein inequality of \cite[Theorem 3.3]{vanzanten}). For any square integrable purely discontinuous martingale $ M, $ writing 
$$ [M]_t = \sum_{ s \le t} (\Delta M_s)^2 , \; \mbox{ and }\;  < M>_t \; \mbox{ for its predictable compensator},$$ 
and putting, for a fixed $a,$ 
$$ H_t^a := \sum_{ s \le t} (\Delta M_s)^2  1_{\{ | \Delta M_s| > a \} } + <M>_t , $$
we have that 
\begin{equation}\label{eq:bernstein}
\P ( M_t^* \geq z , H_t^a \le L ) \le 2 \exp \left( - \frac12 \frac{z^2 }{L} \psi ( \frac{az}{L} ) \right) , 
\end{equation} 
where $ \psi ( x) = ( 1 + x/3) ^{- 1 } $ and $ M_t^* = \sup_{s \le t} |M_s| .$ 

Observing that $ M^{N, 2 } $ and $M^{N, 3 }$ have jumps bounded by $ \frac1N $ and quadratic variation bounded by 
$$ < M^{N, i }>_T \; \le \frac{K T }{ N} ,\;  i = 2, 3 , $$  
we apply \eqref{eq:bernstein} with $ a = \frac1N$ and $ L = \frac{KT}{N} $ to obtain, for $ i=2, 3 , $
$$ \P ( \sup_{ s \le T} | M_s^{N, i } | \geq z ) \le  2 \exp \left( - \frac{1}{2} \frac{z^2 }{KT } \psi ( \frac{z}{KT} ) N \right).$$
Choosing $ z = \frac{\varepsilon KT}{N^{1/4}}, $ for some $0 <  \varepsilon < 1, $  this yields, for $ i=2, 3, $ 
\begin{equation}\label{eq:dev1}
 \P ( \sup_{ s \le T} | M_s^{N, i } | \geq  \frac{\varepsilon KT}{N^{1/4}}  ) \le 2 \exp \left( - \frac12  \varepsilon^2 \psi ( 1) \sqrt{N} KT  \right)  \le 2  \exp \left( - \frac12  \varepsilon^2 \psi ( 1) \sqrt{N}  \right) ,
\end{equation}
where we have used that  $KT \geq 1, $ that $ \psi (  \frac{z}{KT} ) \geq \psi ( 1), $ since  $ \psi $ is decreasing, and $\frac{z}{KT} \le 1.$ 

We now treat $M_t^{N, 1 } .$  In the following we shall apply \eqref{eq:bernstein} with $ a =0.$ Then 
$$ H_T^0 \le \frac{\alpha^2 }{N^2 } \sum_{j=1}^N \int_0^T \int_{\R_+} ( R_{s-}^{N } (j) )^2   1_{\{ z \le \varphi ( U_{s-}^{N} \}} \bM^j ( ds, dz ) + \frac{\alpha^2 K  }{N^2 } \sum_{j=1}^N \int_0^T  ( R_{s}^{N} (j) )^2 ds . $$
Using that $ R_s^{N } (j) \le R_0^{N }(j)  + J_T (j) , $ we get 
$$  \frac{\alpha^2 }{N^2 } \sum_{j=1}^N \int_0^T \int_{\R_+} ( R_{s-}^{N }(j) )^2  1_{\{ z \le \varphi ( U_{s-}^{N} \}} \bM^j ( ds, dz ) 
  \le \frac{ 2 \alpha^2}{N} \left( \frac1N \sum_{ j=1}^N  \left[ (R^{N }_0 (j))^2 +  (J_T (j) )^2 \right] J_T (j) \right)  .
$$
This implies that $ H_T^0 $ possesses exponential moments; that is, there exists some $ \nu > 0 $ such that $ \E e^{\nu H_T^0 } < \infty .$ Moreover, since $ \sup_s \E ( R_{s}^{N } (j) )^2 < \infty , $ $ \E H_T^0 \le 2 \frac{ \alpha^2 K}{N} T  \sup_t \E ( R_{s}^{N } (j) )^2 = C T /N . $ 
Therefore, using  a large deviations upper bound, choosing $c_T= C T+1  ,$ we deduce that for a suitable constant $ C ,$
$$ \P ( H_T^0 \geq c_T / N ) \le C e^{ - c N } .$$
Applying \eqref{eq:bernstein} with $ L = c_T/ N , z^2 =\varepsilon ^2 c_T N^{-1/2}  $ and $a = 0 , $  we deduce from this that 
$$ \P ( \sup_{ s \le T } | M_s^{N, 1 } | \geq \varepsilon \sqrt{c_T} N^{-1/4}    ) \le  C e^{ - \frac12  \varepsilon^2  \sqrt{N} } + C e^{ - c N }\le C e^{ - c  \varepsilon^2  \sqrt{N} }  .$$
Consequently, choosing e.g. $\eta = \varepsilon N^{- 1/4} , $ 
\begin{multline*}
 \P \left( \sup_{s \le T} \left( | U_s^N - u_s |  + | \frac1N \sum_{i=1}^N R_s^{N } (i)  - r_s | \right)\geq N^{-1/4} (C_R)^{T+1}  [ 2\varepsilon + \varepsilon T^{1/2} + 2 KT \varepsilon ] \right)\\
  \le 
C T  e^{ - c   \varepsilon^2  \sqrt{N} }  .
\end{multline*}
Taking $N_0$ such that $  N_0^{-1/4} (C_R)^{T+1}   [ 2 + T^{1/2} + 2 KT  ] \le N_0^{-1/5}  ,$ we obtain the assertion.
\end{proof}

We close this section with the proof of  the second item of Theorem \ref{theo:4}. 

\begin{proof}[Proof of the third item of Theorem \ref{theo:4}] 
Since $  (u^{max}, r^{max}  )$ is locally attracting, $t_1 $ is finite (and deterministic). 
We apply  the first item of Theorem \ref{theo:4} with $ T = t_1 .$ It implies that for all $ N \geq N_0 ,$ 
$$
 \P \left(   | U_{t_1} ^N - u_{t_1} |  + | \frac1N \sum_{i=1}^N R_{t_1} ^{N } (i) - r_{t_1} | \geq   \varepsilon    \right) \le \\
C_{t_1} e^{ - c_{t_1} \varepsilon^2    \sqrt{N}  } ,
$$ 
which implies the assertion. 
\end{proof}

\section*{Acknowledgements}
The authors thank two anonymous referees for useful remarks and careful reading. AG and EL thank the Gran Sasso Science Institute (GSSI) for hospitality and support. This research is part of USP
project {\em Mathematics, computation, language and the brain} and of
FAPESP project {\em Research, Innovation and Dissemination Center for
  Neuromathematics} (grant 2013/07699-0). AG is partially supported by
CNPq fellowship (grant 311 719/2016-3.)


\begin{thebibliography}{99}
\bibitem{BT}
{\sc  Barak O., Tsodyks M.}
\newblock Persistent Activity in Neural Networks with Dynamic Synapses. 
\newblock {\em PLoS Comput Biol 3}(2) (2007), https://doi.org/10.1371/journal.pcbi.0030035


\bibitem{blackman.vigna:18}
{\sc Blackman, D., Vigna, S.}
\newblock Scrambled Linear Pseudorandom Number Generators.
\newblock {\em CoRR}, abs/1805.01407 (2018), http://arxiv.org/abs/1805.01407.

\bibitem{Brillinger}
{\sc Brillinger, D., Segundo, J.P.}
\newblock Empirical Examination of the Threshold Model of Neuron Firing.
\newblock {\em Biol. Cybernetics 35} (1979), 213--220.

\bibitem{ccdr}
{\sc Chevallier, J., Caceres, MJ., Doumic, M., Reynaud-Bouret, P.}
\newblock Microscopic approach of a time elapsed neural model.
\newblock {\em Math. Mod. \& Meth. Appl. Sci.}, 25(14) (2015) 2669--2719.

\bibitem{chorno}
{\sc Chornoboy, E.,  Schramm, L., Karr, A.}
\newblock  Maximum likelihood identification of neural point process systems. 
\newblock {\em Biological Cybernetics} 59 (1988), 265--275.

\bibitem{SusanneEva} 
{\sc Ditlevsen, S., L\"ocherbach, E.}
\newblock  Multi-class oscillating systems of interacting neurons.
\newblock {\em Stoch. Proc. Appl.} {127}, 1840--1869, 2017.

\bibitem{do}
{\sc Duarte, A., Ost, G.}
\newblock A model for neuronal activity in the absence of external stimuli.
\newblock {\em Markov Process. Related. Fields}  22 (2016) 37-52.


\bibitem{vanzanten} 
{\sc Dzhaparidze, K., van Zanten, J.H.}
\newblock  On Bernstein-type inequalities for martingales.
\newblock {\em Stochastic Processes Appl. 93}, No.1 (2001), 109-117.

\bibitem{van}
{\sc Kistler, W.M., van Hemmen, L.} 
\newblock Short-Term Synaptic Plasticity and Network Behavior.
\newblock {\em Neural Computation 11} (1999), 1579--1594.

\bibitem{hrbr}
{\sc Hansen, N., Reynaud-Bouret, P., Rivoirard, V.}
\newblock Lasso and probabilistic inequalities for multivariate point processes.
\newblock {\em Bernoulli}, 21(1) (2015) 83-143.

\bibitem{IW}
{\sc Ikeda, N., Watanabe, S.}
\newblock {\em Stochastic differential equations and diffusion processes}.
\newblock North Holland, 1989.


\bibitem{js}
{\sc Jacod, J., Shiryaev, A.N.}
\newblock Limit theorems for stochastic processes.
\newblock Second edition, Springer-Verlag, Berlin, 2003.

\bibitem{markram96}
{\sc Markram, H., Tsodyks, M.}
\newblock Redistribution of synaptic efficacy between neocortical pyramidal neurons. 
\newblock {\em Nature 382}, 6594 (1996), 807-810.

\bibitem{Mongillo}
{\sc Mongillo, G., Barak, O., Tsodyks, M.}
\newblock Synaptic Theory of Working Memory.
\newblock {\em Science 319}, 1543 (2008).

\bibitem{ogata:81}
{\sc Ogata, Y.}
\newblock {O}n {L}ewis' simulation method for point processes. 
\newblock {\em IEEE Transactions on Information Theory} IT-27 (1981), 23-31.


\bibitem{pat}
{\sc Reynaud-Bouret, P.,  Rivoirard, V., Grammont, F., Tuleau-Malot, C.}
\newblock  Goodness-of-fit tests and nonparametric adaptive estimation for spike train analysis. 
\newblock{\em The Journal of Mathematical Neuroscience (JMN)} 4 (2014), 1--41.

\bibitem{Schmutz}
{\sc Schmutz, V., Gerstner, W., Schwaiger, T. }
\newblock Mesoscopic population equations for spiking neural networks with synaptic short-term plasticity.
\newblock {arxiv.org/abs/1812.09414}, 2018. 

\bibitem{see}
{\sc Seeholzer, A., Deger, M., Gerstner, W.}  
\newblock Stability of working memory in continuous attractor networks under the control of short-term plasticity. 
\newblock{\em bioRxiv 424515} (2018);  doi: https://doi.org/10.1101/424515.



\bibitem{s}
{\sc Sznitman, A.-S.}
\newblock Topics in propagation of chaos.
\newblock In {\em \'{E}cole d'\'{E}t\'e de {P}robabilit\'es de {S}aint-{F}lour
{XIX}---1989}, vol.~1464 of {\em Lecture Notes in Math.} Springer, Berlin,
1991, 165--251.

\bibitem{tm}
{\sc Tsodyks, M., Markram, H.}
\newblock The neural code between neocortical pyramidal neurons depends on neurotransmitter release probability. 
\newblock {\em Proceedings of the National Academy of Sciences 94}(2) (1997),  719-723. 

\bibitem{tm2}
{\sc Tsodyks, M., Pawelzik, K., Markram, H.},
\newblock Neural Networks with Dynamic Synapses.
\newblock {\em Neural Comput. 10}(4) (1998), 821--835.

 
\bibitem{scholar}
{\sc Tsodyks, M., Wu, S.} 
\newblock Scholarpedia, 8(10):3153, (2013).

\bibitem{zucker}
{\sc Zucker, R.S., Regehr, W.G.}
\newblock Short-term synaptic plasticity.
\newblock {\em Annu. Rev. Physiol.} 64, (2002), 355–405.

\end{thebibliography}
\end{document}